\documentclass[12pt]{article}

\usepackage{amsmath, amssymb, latexsym, amscd, amsthm}

\newtheorem{theorem}{Theorem}[section]
\newtheorem{corollary}[theorem]{Corollary}
\newtheorem{lemma}[theorem]{Lemma}
\newtheorem{proposition}[theorem]{Proposition}

\theoremstyle{definition}
\newtheorem{definition}[theorem]{Definition}
\newtheorem{remark}[theorem]{Remark}

\begin{document}


\title{Pairwise disjoint Moebius bands in space}

\author{Olga Frolkina\\
Chair of General Topology and Geometry,\\
Faculty of Mechanics and Mathematics,\\
M.V.~Lomonosov Moscow State University,\\
Leninskie Gory 1, GSP-1,\\
Moscow 119991, Russia\\
E-mail: olga-frolkina@yandex.ru
}

\bigskip

\maketitle

\def\conv{\operatorname{conv}}
\def\Int{\operatorname{Int}}
\def\id{\operatorname{id}}

\begin{abstract}
V.V.~Grushin and V.P.~Palamodov proved in
1962 that it is impossible
to place in $R^3$
uncountably many
pairwise disjoint polyhedra
each homeomorphic to the Moebius band.
We generalize this result in two directions.
First, we give a generalization of this result to tame subsets in $R^N$, $N\geqslant 3$.
Second, we show that in case of $R^3$ the theorem holds
for arbitrarily topologically embedded 
(not necessarily tame)
Moebius bands.


\end{abstract}

\section{Introduction}

Theory of wild embeddings comes back
to papers of L.~Antoine, P.S.~Urysohn,
J.W.~Alexander 
from 1920's.
Let us mention
the reviews
\cite{BC}, \cite{Burgess-75}, \cite{Cernavsky}, \cite{Daverman}
and the
books \cite{DV}, \cite{Keldysh}, \cite{Rushing}
which
contain hundreds of references on the subject.

\begin{definition}
A subset $X\subset R^N$ is called a {\it polyhedron}
if it is a union of a finite collection of 
simplices.
A subset $X\subset R^N$ homeomorphic to a polyhedron
is called {\it tame}
if there is a homeomorphism $h$ of
$R^N$ onto itself such that
$h(X) \subset R^N$ is a polyhedron;
otherwise, $X$ is called {\it wild}.
An embedding $F:X\to R^N$ is called wild
if its image $F(X)$ is a wild subset of $R^N$,
and tame otherwise.
\end{definition}

For cells and spheres,  it is useful to compare
this with
the following definition:

\begin{definition}
A subset $X\subset R^N$ homeomorphic
to $I^k$ or $S^k$ is called {\it flat}
if there exists a homeomorphism $h$
of $R^N$ onto itself such that
$h(X)$ is a $k$-simplex
or
the boundary of a $(k+1)$-simplex.
\end{definition}

From classical results of A.~Schoenflies and L.~Antoine we know that each
arc and each simple closed curve in $R^2$
is flat \cite[II.4]{Keldysh}.
Antoine constructed
first wild arcs in $R^3$ 
in 
\cite[{\bf 54--58}, p.65--70; {\bf 83}, p. 97]{Antoine-diss}.
His first wild arc $ab$ 
\cite[{\bf 54--58}, p.65--70]{Antoine-diss}
consists 
of an arc $ap$ with a sequence of trefoils tied in it and convergent to the end point $p$, 
united with a straight line segment $pb$
(see the picture in \cite[Exercise 2.7.3]{Rushing}
or \cite[Fig.~2.26]{DV}; this example is
often attributed to R.L.~Wilder
because of
the footnote on page 634 of \cite{Wilder};
let us underline that the footnote contains reference to Antoine's Thesis \cite{Antoine-diss}).
Second wild arc \cite[{\bf 83}, p. 97]{Antoine-diss} contains a wild Cantor set (now called Antoine necklace).
Moreover, Antoine constructs an everywhere wild arc in $R^3$ \cite{Antoine-FM}.
Well-known Alexander horned sphere
\cite{Alexander-sphere}
and Antoine-Alexander sphere
\cite[p. 285]{Antoine-CR},
\cite{Alexander},
\cite[{\bf 21}, p.283--284]{Antoine-FM}
provide examples of wild surfaces
in $R^3$.
For $N\geqslant 3$,
each polyhedron $P\subset R^N$ with
${\dim P \geqslant 1}$
can be wildly embedded in $R^N$
(see
\cite{Ivanov-diss}
and \cite{Blankinship};
compare
\cite[{\bf 83}, p. 97]{Antoine-diss}, 
\cite{Antoine-CR},
\cite{Antoine-FM}).
In $R^N$, $N\geqslant 3$, there
exist 
everywhere wild cells and spheres
of all dimensions $1,\ldots , N-1$
(see \cite{Bing-everywhere}, 
\cite[p.86--88, p.73--74]{DV}, 
\cite[Thm. 2.6.1]{Rushing}).
It is known 
that for each $N$ and each $k\neq N-2$
any tame $k$-cell or $k$-sphere
in $R^N$ is flat,
refer e.g. to \cite[Cor.~1.1.2, Thm.~1.1.3, Cor.~1.1.6]{DV}.
For any $N$, each tame arc in $R^N$ is flat
\cite[p.125--126]{Keldysh}.
For $N\geqslant 4$
there exist tame but not locally flat
$(N-2)$-cells and $(N-2)$-spheres in $R^N$
\cite[Example 1.4.2]{DV}.

``Being wild'' should not be considered as a 
strange
exception: in the sense of Baire category,
most knots in $R^3$ are wild \cite[Thm.~1]{Milnor-wild} (see also \cite[Prop. 6.3.8]{DV});
most arcs in $R^3$ are wild \cite{Bothe-wild};
in contrast to this, 
most 2-spheres in $R^3$ are tame \cite{Bothe-wild}.

Classification of wild embeddings is a hard problem.
By \cite[Thm.5.4]{GKB}, the classification problem of embeddings of the Cantor set $\mathcal C$
in $R^3$
is at least as complicated as the classification of countable linear orders.
Using methods of descriptive theory it is shown in \cite{Kulikov} that the classification
problem for wild knots is strictly harder than that for countable structures.

R.H.~Bing proved that 
it is impossible to place in
$R^3$
uncountably many
pairwise disjoint wild closed
surfaces.
Scheme of proof is given in
\cite{Bing-abstr},
and the full proof needs results of
\cite{Bing59} and \cite{Bing-TAMS61};
also, proof of Bing is explained in
\cite[Thm.~3.6.1]{BC}. (The assumption of wildness is essential: concentric spheres 
of all positive radii is an uncountable family
of pairwise disjoint tame surfaces.)
J.R.~Stallings
constructed  \cite{Stallings}
a
family
of
pairwise disjoint wild 2-disks in
$R^3$ which
has a cardinality of continuum.
J.~Martin showed
that
all except countably many disks in such a
collection must be locally tame at their
interior points \cite{Martin} (this result was predicted in
\cite{Bing-TAMS61} where important facts needed
for Martin's arguments are proved).
R.B.~Sher 
modified Stallings' construction so that 
no two
disks of the family 
are ambiently homeomorphic, that
is, no self-homeomorphism of $R^3$ 
can map one disk onto another
\cite{Sher-note}.
For higher dimensions,
Bing 
non-embedding result 
has 
at the moment
only partial
generalizations, see
\cite[Thm. 1, 2]{Bryant},
\cite[Thm. 10.5]{Burgess-75},
\cite[p.383, Thm. 3C.2]{Daverman}.
The results of Stallings and Sher 
were generalized by the author
to the case of arbitrary
$R^N$, $N\geqslant 3$,
and arbitrary perfect compact
space embeddable in a hyperplane
$R^{N-1}\subset R^N$  \cite{Frolkina}, \cite{Frolkina-Thesis}.
During my Conference talk
\cite{Frolkina-Thesis}
I was asked by Professor
Scott Carter 
if it is possible to place incountably many 
pairwise disjoint
Moebius bands in $R^3$.
He informed that in prior years he knew the impossibility proof from Bob Williams (University of Texas).
I became very interested
in this question.
I found a paper of V.V.~Grushin and V.P.~Palamodov 
containing a partial answer
\cite[Thm.~3]{GP}:
{\it 
Let $P$ be a 2-dimensional connected polyhedron which
can be represented as a (finite) union of
$2$-simplices.
In $R^3$ one can place
uncountably many
pairwise disjoint polyhedra
homeomorphic to $P$ if and only if $P$ 
is orientable and each point $p\in P$ has a planar neighborhood.}
Let us underline that for the ``$\Rightarrow $''
direction, each copy of $P$
is assumed to be a polyhedron in $R^3$.
The proof given in \cite{GP}
essentially exploits 
this strong assumption.
In his Zentralblatt review to Grushin-Palamodov paper, H.G.~Bothe
states their
result \cite[Thm.~3]{GP}
in a more general form,
without assumption that
$P$ is a union of $2$-simplices
\cite{Bothe-review}.

Below, I generalize
result of V.V.~Grushin and V.P.~Palamodov in two directions:
1)
I 
replace $R^3$ by arbitrary $R^N$
but assume that the subsets are
tame  (Corollary~\ref{tame-N-impossible});
and 2)
I stay in $R^3$ but allow
the subsets to be wild (Proposition~\ref{M-3}, Corollary~\ref{M-3-criterion}). 
Corollary \ref{M-3-criterion} 
improves Bothe's assertion \cite{Bothe-review}.

\begin{remark}
In \cite{IN}, results of Grushin and Palamodov 
are represented in a much more stronger form than originaly given in \cite{GP}.
Unfortunately, this leads to several mistakes.
\\
1) ``Simple continua'' in terminology
of V.K.~Ionin and Yu.G.~Nikonorov need not be tame;
for example,
wild arc of Antoine-Wilder 
\cite[{\bf 54--58}, p.65--70]{Antoine-diss},   
\cite[Exercise 2.7.3]{Rushing}
is a ``simple continuum''.
\\
2) Theorem~1 in \cite{IN} is false;
wild Antoine-Wilder arc can serve as a counterexample.
(Its ``spaciality'' in terms of \cite{IN} 
can be derived from 
\cite[{\bf 56--57}, p.66--69]{Antoine-diss};
and it is not hard to show that
it is ``thin'' in terminology of \cite{IN}.
I am planning to explain this in detail
in another paper.)
\\
3) 
Let us analyze Theorem~2 of \cite{IN}.
Suppose that $K$ is a wild $(N-1)$-sphere in $R^N$,
where $N\neq 4$.
 In terminology of \cite{IN}, $K$ is ``thick'' 
by 
\cite{Bing-abstr}, \cite[Thm. 10.5]{Burgess-75}, \cite[p.383, Thm. 3C.2]{Daverman}. Assuming that \cite[Thm. 2]{IN}
is true, we conclude that $K$ can not be 
represented as a union
of a finite family of tame 
cells. This hard statement is of course
not contained in \cite{GP}.
(An $(N-1)$-sphere in $R^N$
which is the union of interiors
of tame $(N-1)$-cells is itself locally
flat and hence 
flat by \cite{Brown}.
If the interiors of the cells do not cover 
the sphere, the statement is proved under
strong
additional conditions in
\cite[Thm. 6.4]{CL} 
using A.V.~Chernavskii's result
\cite[Thm.~2]{Cernavsky-union},
\cite{Cernavsky-2006}.)
\end{remark}

\subsection{Conventions and Notation}

By a map we always mean a continuous map.

For any metric compact space $M$,
the space $C(M,R^N)$ 
of all maps $M\to R^N$
endowed with the uniform distance
is known to be a complete separable metric space.

By $\conv (S)$ we denote the convex hull of a set $S\subset R^N$.

An {\it $N$-manifold} is
a separable metric space
each point of which has a neighborhood
homeomorphic to $R^N$.
An {\it $N$-manifold-with-boundary}
is
a separable metric space
each point of which has a neighborhood
whose closure is homeomorphic to 
an $N$-simplex.
(So, any $N$-manifold is an $N$-manifold-with-boundary.)
For an $N$-manifold-with-boundary $M$, its
boundary is denoted by $\partial M$ and
its interior 
$M - \partial M$
by $\Int M $.

By $\overline Y$ we denote the closure of a set $Y$
in a given ambient space.

\section{Families of tame $(N-1)$-surfaces in $R^N$}

In their proof of \cite[Thm. 2]{GP},
Grushin and Palamodov show that if
two polyhedral Moebius bands in $R^3$
consisting of the same number of triangles
could be ``sufficiently close'' to each other,
than they would be orientable.
Our next Proposition is based on the same idea.
We make a more careful investigation;
and only
the ``limit'' Moebius band is assumed
to be tame.

In contrast to Grushin and Palamodov who measure
``nearness'' of two polyhedral Moebius bands by
considering broken lines with vertices 
in centers of the faces,
we use properties of space of maps
$C(M,R^N)$ (inspired by 
the Bing idea, see
\cite[Thm. 3.6.1]{BC}, \cite{Bing-abstr}); 
we thus obtain
Corollary \ref{tame-N-impossible}.

\begin{proposition}\label{tame-N}
Let $M$ be a compact triangulable
$(N-1)$-manifold-with-boundary,
and let $f_1,f_2,\ldots , F$ be its embeddings
into $R^N$ such that:
$F(M)$ is tame;
the sequence $\{f_n\}$ converges to $F$
in $C(M,R^N)$;
and 
$f_n(M)\cap F(M)=\emptyset $ for each~$n$.
Then $M$ is orientable.
\end{proposition}

\begin{proof}
Since $F(M)$ is a tame subset of $R^N$,
there exists a homeomorphism $h:R^N\cong R^N$
such that $h(F(M))$ is a polyhedron.
Using uniform
continuity
of $h$ on a large compact ball 
$B\supset \Int B \supset F(M)$
we see
that the sequence of maps
$\{h\circ f_n\}$ converges to
$h\circ F$ in $C(M, R^N)$.
Therefore we will omit the homeomorphism $h$;
we will assume that $F(M)$ is already a polyhedron in $R^N$,
and the sequence 
$\{ f_n\}$ converges to $F$.

Let us assume that $M$ is non-orientable.
Fix any triangulation $\tau $ of the polyhedron $F(M)$.
There exists a sequence
of $(N-1)$-simplices 
$\Delta _1 , \ldots , \Delta _m$
of $\tau $
such that
each pair of adjacent simplices
(and also the last and the first one)
have a common $(N-2)$-face;
and this sequence does not possess
coherent orientations of its simplices.
We may also assume that
$\Delta _i \neq\Delta _j$ for $i\neq j$.

For each $i=1,\ldots , m$ let $A_i$ be the center of the simplex $\Delta _i$, and
$A_{i,i+1}$ be the center of the $(N-2)$-simplex
$\Delta _i \cap \Delta _{i+1}$.
(By $\Delta _{m+1}$ we mean
$\Delta _1$;
in particular, $A_{m+1} = A_1$,
and $A_{m,m+1}$ is the center of
the $(N-2)$-simplex $\Delta _m\cap \Delta _1$.)
Let $A_iA_{i,i+1}A_{i+1}$ be
the broken line defined as the union
of two segments:
$A_iA_{i,i+1}$ and $A_{i,i+1}A_{i+1}$.
Let us denote
$P_i = F^{-1} (A_i)$ and
$q_i = F^{-1}(A_iA_{i,i+1}A_{i+1})$
for each $i=1,\ldots , m$.

Take a positive number $\delta $ such that 
$$
 \delta < \min_{i=1,\ldots , m}
\{
d(A_i, \partial \Delta _i);
\ 
d(
A_iA_{i,i+1}A_{i+1}, \partial (\Delta _i\cup 
\Delta _{i+1}) )
\} .
$$

Take an integer $K$ such that
$d(f_K,F)<\delta $.

Fix an orientation of $R^N$.
For each $i=1,\ldots , m$ we have
$$
d(f_K (P_i) , A_i) =
d(f_K (P_i) , F(P_i))  < \delta < 
d(A_i, \partial \Delta _i);
$$
moreover,
$f_K(P_i) \notin \Delta _i \subset F(M)$.
Therefore the point 
$f_K(P_i)$
defines the direction of a normal vector 
for the hyperplane containing
the simplex $\Delta _i$.
Moreover, for each $i=1,\ldots , m$
we have
$
d(f_K(q_i), A_iA_{i,i+1}A_{i+1}) = 
d(f_K(q_i) , F(q_i)) < \delta $
and
$f_K(q_i)\cap (\Delta _i\cup\Delta _{i+1})
\subset f_K (M)\cap F(M) = \emptyset $;
hence
the orientations of
$\Delta _i$ and $\Delta _{i+1}$ thus defined
are coherent.
This contradiction finishes the proof.
\end{proof}

\begin{corollary}\label{tame-N-impossible}
It is impossible to place in $R^N$, $N\geqslant 3$
uncountably many
pairwise disjoint
tame 
homeomorphic
images of a compact triangulable
non-orientable ${(N-1)}$-manifold-with-boundary.
\end{corollary}

\begin{proof}
Suppose the contrary.
Let $\mathcal H = \{ h_\alpha : M\to R^N, \alpha \in A\}$
be an uncountable family of
embeddings such that
$h_\alpha (M)\cap h_\beta (M) =\emptyset $
for each $\alpha \neq \beta $, where
$M$ is a compact triangulable
non-orientable $(N-1)$-manifold-with-boundary.

The space
$C(M , R^N)$ is separable.
Hence its un\-count\-able subset $\mathcal H$
has a converging subsequence
of distinct elements
$h_{\alpha _m}$ (the limit map $h=\lim\limits_{m\to\infty }  h_{\alpha _m}$
also  belongs to $\mathcal H$).
To finish the proof, apply Proposition \ref{tame-N}.
\end{proof}

\section{Families of topological Moebius bands in $R^3$}

\begin{definition}
A set $X\subset R^N$  is called
{\it locally tame} at $x\in X$
if there is a neighborhood $U$ of $x$
in $R^N$ and a homeomorphism
$h$ of $\overline U$ into $R^N$ such that
$h(\overline{U} \cap X)$
is a polyhedron.
\end{definition}

By 
\cite{Bing-loc-tame}
or \cite{Moise-loc-tame},
closed locally tame subsets of $R^3$
are tame.

\begin{proposition}\label{M-3}
Let $M$ be a compact
non-orientable
2-manifold-with-boundary.
It is impossible to place
uncountably many 
pairwise disjoint
homeomorphic images of $M$
in
$R^3$.
\end{proposition}

\begin{proof}
It suffices to consider the case 
when $M$ is a (compact) Moebius band
$\mu = I\times I /  (0,t) \sim (1,1-t)$.
Let us assume that $\{ M_\alpha , \alpha \in A\}$
is an uncountable family
of pairwise disjoint subsets of $R^3$, and
$h_\alpha : \mu \cong M_\alpha $
is a homeomorphism for
each
$\alpha \in A$.

Let $\pi : I\times I \to \mu $
be the factorization map.
Represent
$\mu $ as a union of two disks:
$\mu = D_1 \cup D_2 ,
$
where
$$
D_1 = \pi \left(\left[ 0;\frac34\right] \times I\right);
\quad
D_2 = \pi \left(\left(\left[ 0;\frac14\right] \cup
\left[ \frac12 ; 1\right]\right) \times I\right).
$$

Apply \cite[Lemma 2]{Martin} to the uncountable family of 2-disks
$\{h_\alpha (D_1), \alpha \in A\}$.
There exists an uncountable subset
$A_1\subset A$ such that
each disk
$h_\alpha (D_1)$, $\alpha \in A_1$,
is locally tame at each
$x\in \Int h_\alpha (D_1)$.
Again, apply \cite[Lemma 2]{Martin} to the uncountable family of 2-disks
$\{ h_\alpha (D_2), \alpha \in A_1\}$;
there exists an uncountable subset $A_2\subset A_1$ such that
each
$h_\alpha (D_2)$, $\alpha \in A_2$,
is locally tame at each of its interior points.

We obtain an uncountable
set
$A_2$ such that
each disk
$h_\alpha (D_i)$, $\alpha \in A_2$, $i=1,2$,
is locally tame at each of its interior points.
Hence each Moebius band
$h_\alpha (\mu )$, $\alpha \in A_2$, 
is locally tame at each of its interior points.
Replace $\mu $ by 
a smaller Moebius band
$\tilde\mu  = \pi \left(I \times \left[\frac14 ; \frac34\right]\right) \subset \mu$;
we obtain
an uncountable family
of pairwise disjoint
locally tame
Moebius bands
$\{h_\alpha (\tilde \mu ), \alpha \in A_2\}$.
By
\cite{Bing-loc-tame}
or \cite{Moise-loc-tame},
each 
band
$h_\alpha (\tilde \mu )$, $\alpha \in A_2$
is tame.
Application of Corollary~\ref{tame-N-impossible} finishes the proof.
\end{proof}

Recall that a topological space $X$ is called {\it locally planar} if each point $x\in X$ 
has a neighborhood $U$ which can be embedded
in plane.

\begin{definition}\cite[p.165]{GP}\label{or}
Let $P$ be a locally planar
2-dimensional polyhedron (in some $R^N$). 
Take any of its triangulations.
$P$ is called {\it orientable}
if 
one can give
coherent orientations to
all its 2-simplices.
\end{definition}

Note that orientability implies 
local planarity;
but we prefer 
to write ``locally planar and orientable''
in order 
to avoid occasional 
misunderstanding.

\begin{remark}\label{or}
For a locally planar 2-dimensional polyhedron $P$,
the following are
equivalent to the above definition.

1) Let $\hat P$ be the union of all 2-dimensional simplices of $P$. Note that 
$\hat P$ is a 
homogeneous 
non-branching
2-dimensional complex
in the sense of
\cite[p.314, p.462]{Rinow}.
By \cite[43.24]{Rinow}
$\hat P = \cup P_i$
where each $P_i$
is a 2-dimensional pseudomanifold
and $\dim (P_i \cap P_j) \leqslant 0$
for each $i\neq j$.
We say that the polyhedron $P$ is orientable
if each pseudomanifold $P_i$
is orientable.

2) $P$ is called orientable
if it does not contain a 
homeomorphic copy of the Moebius band
(this is taken as a definition in \cite{Bothe-review}).
\end{remark}

Corollary \ref{M-3-criterion}
extends
\cite[Thm. 3]{GP}
and its generalization
stated in \cite{Bothe-review}.
First, let us prove a lemma
(compare
\cite[p.~167, Lemma]{GP}
where the case of 
polyhedra all whose maximal simplices
are 2-dimensional is considered).

\begin{lemma}\label{surfaces}
Let $P$ be a locally planar
2-dimensional polyhedron
(in some $R^N$).
Moreover, suppose that
$P$ is orientable.
Then there exists an orientable 2-manifold-with-boundary $M$ such that $P\subset M\subset R^N$
and $M$ is a polyhedron in $R^N$.
\end{lemma}

\begin{proof}
By \cite[p.167--168, Lemma]{GP}
it suffices 
to find 
a locally planar orientable
2-dimensional polyhedron $\tilde P$
with
$P\subset \tilde P\subset R^N$
such that
each maximal simplex of $\tilde P$
is 2-dimensional.

Fix any triangulation $\tau $ of $P$.
For
each isolated point $A$ of $P$
(0-dimensional maximal simplex of $\tau $), 
replace $A$ with a small triangle
$\Delta $ such that
$A$ is among its vertices
and
$\Delta \cap P = \{ A\} $.
Similarly, if $\sigma $ 
is a maximal 1-dimensional simplex of $\tau $,
replace it with a narrow triangle
$\Delta = \conv (\sigma , B)$,
where $B\notin \sigma $ is a point sufficiently
close to the midpoint of $\sigma $ so that
$\Delta \cap P = \sigma \cap P$.
It can be easily seen that the polyhedron
$\tilde P$
obtained after all replacements 
satisfies the conditions listed above.
\end{proof}

\begin{corollary}\label{M-3-criterion}
Let $P$ be a 2-dimensional polyhedron (in some
$R^N$).
The following are equivalent:
\\
(a)
in $R^3$ one can place
uncountably many
pairwise disjoint polyhedra
homeomorphic to $P$;
\\
(b)
in $R^3$ one can place
uncountably many
pairwise disjoint subsets
homeomorphic to $P$;
\\
(c) 
$P$ is locally planar and orientable.
\end{corollary}

\begin{proof}
$(a)\Rightarrow (b)$ is evident.

Let us prove $(b)\Rightarrow (c)$.
\\
1) Suppose that a point $p\in P$
has no planar neighborhood.
There exist a 2-simplex $\Delta \subset P$
and an arc $\alpha \subset P$ such that
$p\in \Int \Delta  \cap \partial \alpha $.
By \cite{Young},
it is impossible to place uncountably many
disjoint
homeomorphic copies of the ``umbrella''
$\Delta \cup \alpha $ (hence also of $P$) 
in $R^3$.
(V.V.~Grushin and V.P.~Palamodov seem
to be unfamiliar with \cite{Young} while writing \cite{GP}, and they included the proof
of impossibility to place in $R^3$
uncountably many disjoint polyhedra
homeomorphic to the umbrella.)
\\
2) Assume that $P$ is non-orientable.
Then it contains a 
homeomorphic image of the Moebius band $\mu $.
By Proposition \ref{M-3}, it is
impossible to place uncountably many
disjoint
homeomorphic copies of $\mu $ 
(hence also of $P$) in $R^3$.

Finally, $(c)\Rightarrow (a)$
can be proved similarly to \cite{GP},
using Lemma \ref{surfaces}.
Namely,
there exists an orientable polyhedral 
2-manifold-with-boundary $M$ with
$P\subset M \subset R^N$.
It is known that $M$ is homeomorphic
to a sphere with handles and holes;
clearly we can place in $R^3$
uncountably many pairwise disjoint
polyhedra each PL-homeomorphic to $M$.
The inclusion $P\subset M$
can be interpreted as a PL embedding;
thus we obtain
uncountably many
disjoint homeomorphic polyhedral
copies of $P$ in $R^3$.
\end{proof}

\section*{Acknowledgments}

I would like to thank Sergey Melikhov
for useful remarks which
he made while reviewing this paper.

\end{document}